\documentclass{amsart}
\usepackage{note_en}
\newcommand{\DD}{\mathcal{D}}
\newcommand{\VV}{\mathcal{V}}
\newcommand{\Lie}{\mathfrak{L}}
\newcommand{\diag}{\mathrm{diag}}
\newcommand{\Supp}{\mathrm{Supp}}

\begin{document}
\title{Heat Conduction Networks:~Disposition of Heat Baths and Invariant Measure}
\author{Alain Camanes}
\date{\today}
\address{Laboratoire Jean Leray UMR 6629\newline
Université de Nantes\newline
2, Rue de la Houssinière BP 92208\newline
F-44322 Nantes Cedex 03}
\email{camanes-a@univ-nantes.fr}
\date{\today}
\subjclass[2000]{Primary 82C22 ; Secondary 60K35, 60H10}
\keywords{Heat conduction network, invariant probability, support of invariant measures, diffusion process, Asymptotically Strong Feller, Lyapunov equation, controllability, Lasalle's principle, heat conduction network}
\bibliographystyle{alpha}
\renewcommand{\labelitemi}{$\ast$}
\renewcommand{\labelitemii}{$\ast$}
\maketitle
\begin{abstract}
We consider a model of heat conduction networks consisting of oscillators in contact with heat baths at different temperatures. Our aim is to generalize the results concerning the existence and uniqueness of the stationnary state already obtained when the network is reduced to a chain of particles. Using Lasalle's principle, we establish a condition on the disposition of the heat baths among the network that ensures the uniqueness of the invariant measure. We will show that this condition is sharp when the oscillators are linear. Moreover, when the interaction between the particles is stronger than the pinning, we prove that this condition implies the existence of the invariant measure.
\end{abstract}

\section{Definitions and Results}

\subsection{The motivations}
We consider an arbitrary graph. At each vertex of this graph, there is a particle interacting with the substrate and with its neighbours. Among these particles, some are linked to heat baths; an Ornstein-Uhlenbeck process models this interaction. Given this graph, we establish conditions on the disposition of the heat baths that entails existence and uniqueness of the invariant measure.

When the graph is reduced to a chain,  each extremal particle is connected to a heat bath. This model has been studied in~\cite{EPRB99a,EPRB99b,EH00,RBT02}. The uniqueness of the invariant measure is obtained using controllability properties. This property is deeply connected to the geometry of the chain:~the behaviour of the extremal particle entails the behaviour of its neighbour and so on\ldots The existence of the invariant measure when the interaction is stronger than the pinning has also been obtained. These results were used in~\cite{C07} to solve some variations of this model developed in~\cite{BO05} on the one side and~\cite{LS04} on the other.

To avoid the particular geometry of the chain, we work with general networks. These heat conduction networks have been introduced in~\cite{MNV03} and~\cite{RB03}. Let us notice that an Ornstein-Uhlenbeck process is the sum of a damping term and an excitation term. To understand the effect of each of these quantities, we will not suppose that the heat baths have non-negative temperatures. A recent work of~\cite{BLLO08} uses this kind of result to prove the existence of a self-consistent temperature profile. We will see that their results are closely related to the geometry of the network they consider. First we introduce the model. Then we will state our main results linking existence and uniqueness of the invariant measure to the disposition of the heat baths. Intuitively, the existence and the uniqueness of the invariant measure is related to the disposition of the damped particles, i.e. the particles interacting with a heat bath. The disposition of the heat baths at \emph{non-negative} temperatures has an effect on the regularity of the invariant measure.

\subsection{The model}
Let $G=(\VV,\sim)$ be a connected graph with vertex set $\VV$. Two vertices $i,\, j\in\VV$ are nearest neighbours if there is an edge between them: $i\sim j$. Every node $i\in \VV$ holds a particle of unit mass described by its position $q_i$ and momentum $p_i$. Each particle is pinned and interacts with its neighbours. The \emph{pinning} potential $V$ and the \emph{interaction} potential $U$ are both assumed to be convex polynomial functions. The Hamiltonian of this system is the sum of kinetic and potential energies
$$H(q,p)=\sum_{i\in \VV}\left(\frac{1}{2}p_i^2+V(q_i)+\frac{1}{2} \sum_{j\sim i}U(q_i-q_j)\right).$$

Among the particles, some of them are damped; the non-empty subset $\DD\subset\VV$ is called the \emph{damped} set. Among the damped particles, there is a non-empty subset $\partial\VV\subset\DD$ called the \emph{boundary} set. Each particle $i\in\partial\VV$ of the boundary set is connected to a heat bath at (non-negative) temperature $T_i$. The excitation transmitted by the heat bath is modeled by a Brownian motion. Thus, the dynamics of the heat conduction network is described by the Hamiltonian system of stochastic differential equations, for any $i\in\VV$,
\begin{equation}\label{eq:hcn}
\left\{\begin{array}{ccl}
dq_i&=&p_i\ dt,\\
dp_i&=&-\partial_{q_i}H\ dt-p_i\indicatrice{i\in\DD}\ dt+\sqrt{2T_i}\indicatrice{i\in\partial\VV}\ dB_i,
         \end{array}
\right.
\end{equation}
where $\{B_i,\,i\in\partial\VV\}$ are independent Brownian motions.
\begin{remark}
 Damped particles can be considered as connected to heat baths at possibly null temperatures.
\end{remark}

In the following we write $N=|\VV|$ the number of particles and $n=2N$. We study the diffusion $Z=(q,p)$, solution of the stochastic differential equation~\eqref{eq:hcn} via its semigroup $(P_t)$. The adjoint of $P_t$ acts on probability measures:~for any measurable set $A$ and probability measure $\mu$,
$$P_t^\star\mu(A)=\int_{\R^n}P_t(z,A)\, \mu(dz).$$
Finally, $\LL$ denotes the generator of the diffusion~\eqref{eq:hcn}. Recall that for any smooth function $f$ in the domain $\DD_\LL$ of $\LL$,
\begin{eqnarray*}
 \LL f &=&\sum_{i\in\VV}\partial_{p_i}H\partial_{q_i} f-\partial_{q_i}H\partial_{p_i}f-\sum_{i\in\DD}p_i\partial_{p_i}f+\sum_{i\in\partial\VV}T_i\partial_{p_i}^2f\\
  &=&\{H,f\}-\sum_{i\in\DD}p_i\partial_{p_i}f+\sum_{i\in\partial\VV}T_i\partial_{p_i}^2f,\\
\end{eqnarray*}
where $\{\cdot,\cdot\}$ denotes Poisson bracket. The operator $\LL^\star$ is the formal adjoint of $\LL$.

We are interested in the existence and uniqueness of a stationary state for heat conduction networks, i.e. in probability measures $\mu$ such that for any $t\geq 0$, $P_t^\star\mu=\mu$, or equivalently, $\LL^\star\mu=0$.

\begin{figure}[ht!]
 \begin{center}
  \input{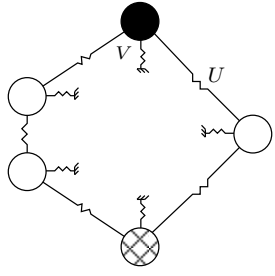}
 \end{center}
\caption{Example of a heat conduction network~:~Damped particles are striped, boundary particles are drawn in black}\label{fig:ex}
\end{figure}

\subsection{The main results}
First of all we will recall and complete some results on the harmonic case, when $U(q)=V(q)=q^2/2$. In this setting, the stochastic differential equation~\eqref{eq:hcn} describing the system is linear.

In the following we will denote $(e_i)_{i\in\{1,\ldots,n\}}$ the canonical base of $\R^n$. A natural space occurring in heat conduction networks is the \emph{controllability space} (see e.g. \cite{W79}). Let $M$ be an $n\times n$ matrix and $\II\subset\VV$ a subset of particles. We will denote $\EE_{M,\II}$ the smallest vector space spanned by $(e_{i+N},\,i\in\II)$ stable by $M$, i.e.
$$\EE_{M,\II}=\mathrm{Span}\left\{M^ke_{i+N},\,i\in\II,\,k\in\N\right\}.$$

We say that the graph is \emph{asymmetric} if $\dim\,\EE_{M,\DD}=n$, where $M$ will be defined in equation~\eqref{eq:harmonic_hcn}. Examples suggest that this condition is related to an asymmetric disposition of the damped particles in the network. Our first theorem states that the existence and uniqueness of the invariant measure is determined by the position of the damped particles.
\begin{theorem}\label{thm:unique_harmonic}
When the potentials are harmonic, there exists a unique invariant measure if and only if
the network is asymmetric, i.e. $\dim\,\EE_{M,\DD}=n$.

When this condition is not satisfied, i.e. $\dim\,\EE_{M,\DD}\neq n$, there exists a quantity invariant with respect to the Hamiltonian flow and the invariant measure is not unique.
\end{theorem}
Let us notice that this uniqueness condition was already known (see~\cite{EZ04}). We give here a new proof using completeness and provide an explicit quantity invariant by the Hamiltonian flow.

The shape of the support of the invariant measure is then described by the position of the heat baths.
\begin{theorem}\label{thm:support_harmonic} Assume there exists a unique invariant measure. The invariant measure has full support if and only if $\dim\,\EE_{M,\partial\VV}=n$.
\end{theorem}

Then we will be interested in the anharmonic case where potentials are only supposed to be convex polynomials. We will first state a theorem where all the damped particles are excited, i.e. $\DD=\partial\VV$. Hörmander's condition will be defined precisely in Section~\ref{section:hormander}. The following theorem will be a straightforward consequence of the weak controllability result obtained in~\cite{H05}.
\begin{theorem}\label{thm:unique_control}
If $\DD=\partial\VV$ and Hörmander's condition is satisfied, there exists at most one invariant measure.
\end{theorem}

Finally, in the more general setting where $\partial\VV\subset\DD$, we will provide a condition (see Section~\ref{section:lasalle}, Condition~\eqref{cond:stability}) on the disposition of the damped particles that entails the uniqueness of the invariant measure. We will show that this condition reduces to the asymmetric disposition of the heat baths when the potentials are harmonic. This condition will be established using a dynamic description of the diffusion through Lasalle's principle. This condition~\eqref{cond:stability} states that when the damped particles are fixed, the system can only be in the equilibrium position. Recall that $H$ is convex, thus $H$ reaches its minimum at an \emph{equilibrium point} $c_0$.
\begin{theorem}\label{thm:unique_lasalle}
When the diffusion is asymptotically strong Feller at $c_0$ and the stability condition~\eqref{cond:stability} is satisfied, the heat conduction network can have at most one invariant measure.
\end{theorem}
Notice that this condition does not dependent on the temperatures of the heat baths.  We will discuss the asymptotic strong Feller condition in Appendix~\ref{section:asf}. We will prove that this regularity condition is satisfied in the harmonic setting as soon as the graph is asymmetric. We could not reach the generality of anharmonic potentials but we expect that this property is still true in this setting.

The importance of the stability condition is strengthened by the following fact about existence of the invariant measure. In~\cite{RBT02}, it is shown that when the pinning potential is weaker than the interaction potential and the network is a chain, there exists an invariant measure. To generalize the method to general networks, we will also have to suppose that the stability condition is satisfied for a limit Hamiltonian (see Section~\ref{section:existence}, Condition~\eqref{cond:rigidite}). We will call this condition the \emph{rigidity condition}. Let $u$ (resp. $v$) be the degree of the polynomial $U$ (resp. $V$).
\begin{theorem}\label{thm:existence}
If the rigidity condition~\eqref{cond:rigidite} is satisfied and $u\geq v$, i.e. the interaction is stronger than the pinning, then there exists a unique invariant measure.
\end{theorem}

Section~\ref{section:harmonic} is devoted to the proofs of Theorems~\ref{thm:unique_harmonic} and~\ref{thm:support_harmonic}. In Section~\ref{section:uniqueness}, we will briefly prove Theorem~\ref{thm:unique_control} using results obtained by M.~Hairer in~\cite{H08}. Then, we will present Lasalle's principle and the associated stability condition. Finally, we will present the proof of the existence of the invariant measure when interaction is stronger than pinning in Section~\ref{section:existence}.

\section{The harmonic case}\label{section:harmonic}
When the potentials $U,\,V$ are harmonic, the system~\eqref{eq:hcn} of stochastic differential equations is linear. We describe the network using the adjacency matrix $\Lg=(\dg_{i\sim j})_{i,j}$ and the degree matrix $D=\diag(\sum_{j\sim i}1)$. The Laplace operator $\Dg$ on the graph $G$ is $\Dg=D-\Lg$. We denote $\Gg=I+\Dg$, $I_\DD=\diag(\indicatrice{i\in\DD})$ and $T_{\partial\VV}=\diag(\sqrt{2 T_i}\indicatrice{i\in\partial\VV})$. Then, the stochastic differential equation~\eqref{eq:hcn} can be written:
\begin{equation}\label{eq:harmonic_hcn}
dZ_t=MZ_t\,dt+\sg\,dB_t,
\end{equation}
with
$$M=\left(\begin{array}{cc}
0&I\\
-\Gg&-I_\DD
\end{array}\right),\qquad
\sg=\left(\begin{array}{cc}
0&0\\
0&T_{\partial\VV}
\end{array}\right).$$

\subsection{Existence and Uniqueness:~First part of Theorem~\ref{thm:unique_harmonic}}
Usually, the existence of the invariant measure is obtained via a compactness property. We provide a new proof in the harmonic case using the fixed point theorem on complete spaces. We could not generalize this argument to a more general setting.

We say that a matrix $M$ is \emph{stable} if its eigenvalues have strictly negative real part. We begin with two linear algebra lemmas proving that when the matrix $\sg$ is null, the noise-free system is contracting if and only if the asymmetry condition $\dim\,\EE_{M,\DD}=n$ is satisfied.

\begin{lemma}[Projection]\label{lem:projection} Let $q\in\R^N$ and $\dim\,\EE_{M,\DD}=n$. If for every $k\in\N$, $I_\DD\Gg^kq=0$ then $q=0$.
\end{lemma}

\begin{proof}
Let $v\in \EE_{M,\DD}$. There exist $\lg_{ik}$ such that $v=\sum_{i\in\DD,k}\lg_{ik}\Gg^k e_i$. Thus,
\begin{eqnarray*}
\langle q,v\rangle&=&\langle q,\sum_{i\in\DD,k}\lg_{ik}\Gg^k e_i\rangle\\
  &=&\sum_{i\in\DD,k}\lg_{ik}\langle\Gg^k q,e_i\rangle\\
  &=&0
\end{eqnarray*}
and $q\in \EE_{M,\DD}^\bot$. Finally, since $\EE_{M,\DD}^\bot=\{0\}$, we get $q=0$.
\end{proof}

\begin{lemma}[Stability]\label{lem:stability}
If $\dim\,\EE_{M,\DD}=n$ then $M$ is stable. 
\end{lemma}

\begin{proof} We have to prove that under the asymmetry condition, the eigenvalues of $M$ have strictly negative real part.\\
The eigenvector $(q,p)\neq 0$ associated to the eigenvalue $\lg$ satisfies
$$\left\{\begin{array}{rcl}
p&=&\lg q\\
-\Gg q-I_\DD p&=&\lg p.
         \end{array}\right.$$
Replacing $p$ by $\lg q$ and doing scalar product with $q$, we get
$$\lg^2\|q\|^2+\lg\|I_\DD q\|^2+\langle\Gg q,q\rangle=0.$$
Thus, the eigenvalues $\lg$ are either negative or conjugated (with strictly negative real part). However, the real part of $\lg$ is null if and only if $\|I_\DD q\|=0$. Thus $I_\DD p=0$, and $I_\DD\Gg q=0.$ To obtain a contradiction, we multiply the first equation by $\Gg$ to obtain $I_\DD\Gg p=0$. Then, by induction on $k$, $I_\DD \Gg^kq=I_\DD\Gg^kp=0$.

Finally, using projection Lemma~\ref{lem:projection}, $q=0$ and $p=0$. Since this is impossible, $M$ is stable.
\end{proof}

We now use a completeness result to prove the existence statement of Theorem~\ref{thm:unique_harmonic}.
Let $(X,d)$ be a Polish space, $\PP$ the set of probability measures on the Borel $\sg$-field of $X$. For any $\mu,\,\nu\in\PP$, we define
\begin{equation}
W(\mu,\nu)=\inf_{\GG\in\CC(\mu,\nu)}\int\int d(x,y)\,\GG(dx,dy),
\end{equation}
where $\CC(\mu,\nu)$ is the set of couplings of $(\mu,\nu)$. The functional $W$ defines the Wasserstein distance on the set $\PP_1$ of probability measures such that for one $z_0\in X$, $W(\dg_{z_0},\mu)<+\infty$. Let us recall (see e.g.~\cite{B07}) that $(\PP_1,W)$ is a complete space. In the following, we denote $d(\mu,\nu)=W(\mu,\nu)$.

\begin{proof}[Proof of the first part of Theorem~\ref{thm:unique_harmonic}]
We first prove a contraction inequality for Dirac measures. Indeed, using a trivial coupling, there exists a non-negative constant $\ag_0$ such that for any starting points $x$ and $y$:
\begin{eqnarray*}
\|P_t^\star\dg_x-P_t^\star\dg_y\|&=&\inf_{\GG\in\CC(P_t^\star\dg_x,P_t^\star\dg_y)}\int\int \|u-v\|\,\GG(du,dv)\\
  &\leq&\int\int \|u-v\|\,P_t(x,du)\otimes P_t(y,dv)\\
  &=&\E\|Z_t^x-Z_t^y\|\\
  &\leq&\|e^{Mt}x-e^{Mt}y\|\\
  &\leq&e^{-\ag_0 t}\|x-y\|,
\end{eqnarray*}
since stability Lemma~\ref{lem:stability} ensures that the matrix $M$ is stable.

We finally use a classical argument (see~\cite[Theorem~2.5]{HM06sg}) to control the distance between two distinct probability measures $\mu,\,\nu$
$$
\|P_t^\star\mu-P_t^\star\nu\|\leq e^{-\ag_0t}\|\mu-\nu\|.$$
Thus, there exists $\ag\in(0,1)$ such that
$$\|P_t^\star\mu-P_t^\star\nu\|\leq\ag\|\mu-\nu\|.$$
The functional $\mu\mapsto P_t^\star\mu$ is contracting on a complete space. Thus, there exists a unique invariant measure $\mu$ such that
$$P_t^\star\mu=\mu,$$
and $\mu$ is the unique invariant measure of diffusion~\eqref{eq:hcn}.
\end{proof}

\begin{example}
One can easily check that the network described in figure~\ref{fig:5atoms} satisfies the asymmetry condition but does not satisfy the condition described in~\cite[Section 2.3, 2.4]{MNV03}. The damped particles are drawn in black.
\begin{figure}[ht!]
 \begin{center}
  \input{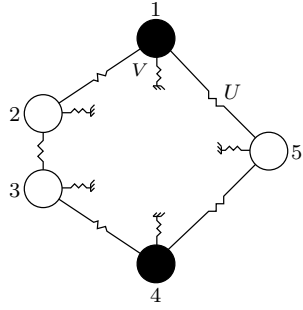}
 \end{center}
\caption{An example of an asymmetric network with $5$ particles}\label{fig:5atoms}
\end{figure}

\end{example}

\subsection{The non-uniqueness:~Second Part of Theorem~\ref{thm:unique_harmonic}}\label{section:unique_harmonic}
\begin{lemma}\label{lem:k} If the graph is not asymmetric, i.e. $\dim\,\EE_{M,\DD}\neq n$, then there exists a quantity $K$ invariant by the Hamiltonian flow and independent of the behaviour of the particles on the boundary set $\partial\VV$.
\end{lemma}

\begin{proof} We look for a quantity $K$ such that
$$K(q,p)=\ag\langle z,q\rangle^2+\langle z,p\rangle^2,$$
where $\langle z,q\rangle=\sum_{i\in\VV}z_iq_i$ denotes the usual scalar product, $\ag$ is a real and $z$ a vector. We look for $\ag,\,z$ such that $K$ is independent of $(p_i,\,i\in\DD)$ to obtain $\LL_T K=0$ and such that $\{H,K\}=0.$ Since $\Gg$ is symmetric, we easily obtain
$$\{H, K\}=2\langle z,p\rangle\langle\ag z-\Gamma z,q\rangle.$$
Since  $\dim\,\left(\EE_{M,\DD}^\bot\right)\geq 1$, $\Gg$ has an eigenvector $z\in \EE_{M,\DD}^\bot$ with eigenvalue $\ag\in\R$. Then  $K(q,p)$ is independent of $p_i,\,i\in\DD$.
\end{proof}

\begin{proof}[Proof of the second part of Theorem~\ref{thm:unique_harmonic}]
For any smooth function $f$ in $\DD_{\LL^\star}$, we decompose the formal adjoint of $\LL$,
\begin{eqnarray*}
\frac{\LL^\star e^f}{e^f}&=& -\{H,f\}+ \sum_{i\in\DD}\left\{1+p_i\partial_{p_i}f\right\}+\sum_{i\in\partial\VV}T_i\left(\partial_{p_i}^2f+(\partial_{p_i}f)^2\right)\\
  &=:&-\{H,f\}+\LL_Tf.
\end{eqnarray*}
Let $\mu$ be an invariant measure. Using Lemma~\ref{lem:k}, there exists a polynomial function $K$ such that for any $\gg>0$, $\mu_\gg$ is invariant where
$$\mu_\gg(dz)=e^{-\gg K(z)}/Z_\gg\,\mu(dz)$$
and $Z_\gg$ is the normalizing constant.\end{proof}

\begin{example}
When the network is the diamond described in figure~\ref{fig:diamant} (the damped particles are drawn in black), we obtain the counter-example given in~\cite[Section~2.3]{MNV03}. Indeed, an eigenvector is $(0,1,0,-1)$ with eigenvalue $3$. Thus, the quantity
$$K(q,p)=3\frac{(p_2-p_4)^2}{2}+\frac{(q_2-q_4)^2}{2}$$
is invariant by the Hamiltonian flow and independent of $(q_1,q_3,p_1,p_3)$.
\begin{figure}[ht!]
 \input{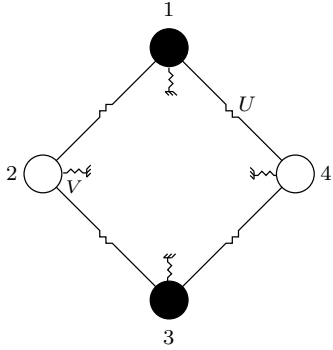}
 \caption{An example of a non-asymmetric heat conduction network}\label{fig:diamant}
\end{figure}

\end{example}

\subsection{Support of the invariant measure:~Proof of Theorem~\ref{thm:support_harmonic}}
To complete the study of the invariant measure, we propose a characterization of its support. Since the pinning and coupling potentials are harmonic, let us recall that the diffusion process starting from $z\in\R^n$ can be written at time $t>0$ by
$$Z_t^z=e^{Mt}z+\int_0^te^{M(t-s)}\sg\,dB_s.$$
Thus, the covariance $K_t$ of the Gaussian process $(Z_t^z)_t$ is $K_t=\int_0^te^{M(t-s)}\sg\sg^\star e^{M^\star(t-s)}\,ds,$ where $M^\star$ denotes the adjoint of $M$. Thus, $(K_t)_t$ satisfies the differential equation
$$\partial_tK_t=\sg\sg^\star+MK_t+K_tM^\star.$$

\begin{proof}[Proof of Theorem~\ref{thm:support_harmonic}]
Let us suppose that $\dim\,\EE_{M,\DD}=n$. Then, uniqueness Theorem~\ref{thm:unique_harmonic} states that the invariant measure is unique. Let us consider the $n\times n$ matrix $Q$ solution of \emph{Lyapunov equation}
 \begin{equation}\label{eq:lyapunov}
 MQ+QM^\star=-\sg\sg^\star.
 \end{equation}
Then, if $K_0=Q$, for any $t\geq 0$, $K_t=Q$ and $Q$ is the covariance of the invariant measure. Then, the matrix $Q$ can be written as
$$Q=\int_0^\infty e^{Mt}\sg\sg^\star e^{M^\star t}\,dt.$$
Moreover, using Lemma~2.3 in \cite{SZ70},
$$\mathrm{rank}(Q)=\dim\,\EE_{M,\partial\VV}.$$
That finishes the proof of Theorem~\ref{thm:support_harmonic}.
\end{proof}

Let us notice that the uniqueness of this solution using Theorem~2.2 in \cite{SZ70} gives an alternative proof of uniqueness Theorem~\ref{thm:unique_harmonic}. Moreover, in~\cite{EZ04}, Lyapunov equation~\eqref{eq:lyapunov} is used to compute the direction of the heat current in simple networks.
\begin{remark}
If $\dim\,\EE_{M,\partial\VV}\neq n$ then the semigroup is not strong Feller. Indeed, the linearity of this equation implies that, for any $z\in\EE_{M,\partial\VV}^\bot$, $Z_t^z\in\EE_{M,\partial\VV}^\bot$. Thus, $P_t(z,\EE_{M,\partial\VV}^\bot)=\indicatrice{\EE_{M,\partial\VV}^\bot}(z)$ is not continuous. This remark has to be linked with the full support property of the invariant measure in Theorem~\ref{thm:unique_control}.
\end{remark}

\subsection{Example}
To end this section, we present a heat conduction network where all the particles are damped and the unique invariant measure has a degenerate support.

In the case described by~figure~\ref{fig:cex}, the network is made of $6$ particles. Tedious but straightforward computations give $\dim\,\EE_{M,\DD}=12$ but $\dim\,\EE_{M,\partial\VV}=10$. The damped particles are represented by striped circles, the boundary set by black ones.
\begin{figure}[ht!]
\begin{center}
\input{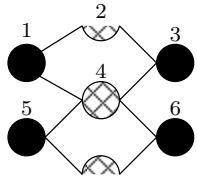}
\end{center}
\caption{An example of a heat conduction network with a unique invariant measure with degenerate support}\label{fig:cex}
\end{figure}

\begin{remark}
In view of \cite{BLLO08}, it could be interesting to wonder if there exists a heat conduction network where the solution $Q$ to the Lyapunov equation satisfies $Q_{i+N,\,i+N}=0$ for some $i\in\VV$. In this case, under the invariant measure, particle $i$ would be motionless. We couldn't find such a graph.
\end{remark}

\section{Uniqueness of the invariant measure}\label{section:uniqueness}
As we could see in the previous section, when the potentials are harmonic, the linearity allows to do explicit computations and we can describe precisely the invariant measures. When potentials are anharmonic, such a description is no more possible. However, we can generalize uniqueness results in the following way. Theorem~\ref{thm:unique_control} relies on controllability and regularity results established via Hörmander's condition. Theorem~\ref{thm:unique_lasalle} relies on Lasalle's principle. The regularity hypothesis are weaker and the temperature conditions are relaxed. As in the harmonic setting, the uniqueness of the invariant measure is related to the damped particles only.

\subsection{Hörmander's condition and uniqueness}\label{section:hormander}
In this subsection, we summarize what can be done using control. We use the strategy developed in~\cite{H05} to prove the uniqueness of the invariant measure when all the damped particles are linked to a heat bath, i.e. $\DD=\partial\VV$. We recall the following regularity property of semigroups and differential operators.

Let $\LL^\star$ be the formal adjoint of the generator $\LL$. Let $r$ be the cardinality of the boundary set $\partial\VV$, $a_0\in\R$ and $X_0,X_1,\ldots,X_r$ be the vector fields such that
$$\LL^\star=a_0+X_0+\sum_{i=1}^r X_i^2.$$
We define inductively the \emph{Lie algebra} $\Lie(z)$ at point $z\in\R^n$ as the algebra generated by
$$\{X_i\}_{i=1,\ldots,r}\,,\,\{[X_i,X_j]\}_{i,j=0,\ldots,r}\,,\,\{[X_i,[X_j,X_k]]\}_{i,j,k=0,\ldots,r}\,,\,\cdots$$
We say that \emph{Hörmander's condition} is satisfied if for every $z\in\R^n$ the lie algebra $\Lie(z)$ has full rank, i.e. $\dim\,\Lie(z)=n$.

If Hörmander's condition is satisfied then the semigroup and the invariant measures have smooth densities with respect to Lebesgue measure (see e.g.~\cite[Corollary~7.2]{RB06}).
\begin{remark}
 When the potentials are harmonic, Hörmander's condition is equivalent to an asymmetric disposition of the heat baths. More precisely, for any $z\in\R^n$, $\dim\,\EE_{M,\partial\VV}=\dim\,\EE_{M,\DD}=\dim\,\Lie(z)=n$.
\end{remark}

\begin{lemma}
When all the temperatures are equal to $1$, the Gibbs measure $\mu_H$ is a full support invariant measure, where
$$\mu_H(dz)=\tfrac{1}{Z}e^{-H(z)}\,dz$$
and $Z$ is the normalizing constant.
\end{lemma}

\begin{remark}
To obtain $\LL^\star\mu_H=0$ in this lemma, it is essential that all the damped particles are excited.
\end{remark}

Theorem~\ref{thm:unique_control} is then a straightforward consequence of the following theorem.
\begin{theorem}[see~\cite{H08}, Theorem~4.1] If Hörmander's condition is satisfied and the Hamiltonian $H$ has compact level sets then the diffusion can have at most one invariant probability measure.
\end{theorem}

\begin{remark}
 It is easy to see from M.~Hairer's proof and using Stroock-Varadhan's Support Theorem that the invariant measure has full support.
\end{remark}

\subsection{Lasalle's principle:~Uniqueness of the invariant measure}\label{section:lasalle}
In this section we relate the uniqueness of the invariant measure to the disposition of the damped particles. When the potentials are harmonic this condition is optimal. Contrary to Section~\ref{section:hormander} we do not suppose that all the damped particles are excited, $\partial\VV\subset\DD$.

This method is based on the contraction properties of the noise-free dynamics used in the lectures of J.~Mattingly~\cite{M07} to prove uniqueness of the invariant measure for Stochastic Navier-Stokes equations. Our aim is to prove that the contraction point is in the support of every invariant ergodic measure. Then, since ergodic measures have disjoint supports, we prove that the invariant measure is unique. To control the deterministic diffusion, we introduce Lasalle's principle.

\medskip

We would like to notice that in this section we just suppose that the semigroup is asymptotically strong Feller at the equilibrium point $c_0$ (see Appendix~\ref{section:asf} for more details).

\subsubsection{Lasalle's principle and Stroock-Varadhan's theorem}
In this section we recall Lasalle's principle. This principle is a generalization of Lyapunov's method. When the derivative of the Lyapunov function is not \emph{definite} negative, this principle says that the solutions of a differential system have an attractive point~(see~\cite[p.~198]{S99}). In the sequel, we write $\dot{H}(z_t)=\partial_t H(z_t)$.

\begin{definition}[Invariant set]
A set $A\subset\R^n$ is called \emph{invariant} if all the trajectories starting from $A$ stay in $A$, i.e. for any $z_0\in A$, for all $t\geq t_0$,
$$z_t^{z_0}\in A.$$
\end{definition}

\begin{theorem}[Lasalle's principle, see~\cite{S99}, Proposition~5.22]
Suppose there exists a function $H:\R^n\to\R_+$ of class $\mathcal{C}^1$ satisfying the following conditions: for all $a>0$,
\begin{enumerate}
\item $\Og_a=\{z;\ H(z)\leq a\}$ is bounded,
\item $\dot{H}\big|_{\Og_a}\leq 0$.
\end{enumerate}
We denote 
$$S=\left\{z\in\Og_a;\ \dot{H}(z)=0\right\}$$
and we consider the biggest invariant subset $A$ of $S$. Then, for any $z_0\in \Og_a$,
$$z_t^{z_0}\xrightarrow[t\to\infty]{} A.$$
\end{theorem}

\begin{remarks}\begin{itemize}
\item When $S$ is reduced to an equilibrium point and $\Og_a\uparrow \R$, this principle implies that for any starting point, the trajectory converges to the equilibrium point.
\item As the function $\dot{H}$ is non-positive on $\Og_a$, $\Og_a$ is invariant.
\item In the Stochastic Navier-Stokes equations (see~\cite{M07}), one can compute the speed of decrease to the equilibrium point of the deterministic system via a Gronwall inequality. In the heat conduction networks we couldn't establish such a result.
\end{itemize}
\end{remarks}

Finally, we recall Stroock-Varadhan's support Theorem. We consider a stochastic system
$$dZ_t=F(Z_t)\ dt+\sg(Z_t)\circ\ dB_t,$$
where $F,\ \sg$ are smooth functions and $\circ$ denotes the Stratonovitch integral. For every $z_0\in\R^n,\, t_0>0$, let
$$\mathcal{S}_{t_0,z_0}=\left\{z_{t_0};\ \exists\psi\in\mathcal{C}^-,\ z_t=z_0+\int_0^t F(z_s)\ ds + \int_0^t \sg(z_s)\psi(s)\ ds\right\},$$
where $\mathcal{C}^{-}$ stands for the set of piecewise continuous functions from $[0,\infty)$ to $\R^n$.\\
We recall that the \emph{support} $\Supp\,\mu$ of a measure $\mu$ is the set of points $z\in\R^n$ such that for any ball $\BB(z_0,\eg)$ of radius $\eg>0$ centred in $z_0$, $\mu\left(\BB(z_0,\eg)\right)>0$.

\begin{theorem}[Support Theorem, see~\cite{SV72}, Section~5]\label{thm:SV}
Using the previous notations,
$$\Supp\,P_{t_0}(z_0,\cdot)=\overline{\mathcal{S}_{t_0,z_0}}.$$
\end{theorem}

\subsubsection{Uniqueness of the invariant measure}

Recall that the Hamiltonian $H$ has a unique minimizer $c_0$. By adding a constant term, we can assume that $H(c_0)=0$. Diffusion~\eqref{eq:hcn} is said to satisfy \emph{stability condition} if the deterministic system
\begin{equation}\label{cond:stability}
\left\{\begin{array}{ccl}
\dot{q}_i&=&p_i\\
\dot{p}_i&=&-\partial_{q_i}H\\
p_i\indicatrice{i\in\DD}&\equiv&0\\
\end{array}\right.
\end{equation}
has a unique solution given by $z_t\equiv c_0$.

\medskip

In our way to prove Theorem~\ref{thm:unique_lasalle}, we are going to show the following theorem.
\begin{theorem}\label{thm:equilibre}
If the stability condition~\eqref{cond:stability} is satisfied then for every invariant measure $\mu$
$$c_0\in \Supp\,\mu.$$
\end{theorem}

\medskip

Since $H$ is continuous at $c_0$, we first notice that for all $\eg>0$ there exists $\eta>0$ such that
$$K_\eta:=\left\{z;\ H(z)< \eta\right\}\subset \BB(c_0,\eg).$$
Theorem~\ref{thm:equilibre} is proved using the following way:~from Lasalle's principle, every solution of the deterministic system associated to~\eqref{cond:stability} goes to $c_0$. Hence, it hits $K_\eta$. As soon as the dynamical system enters $K_\eta$, it stays there, since $K_\eta$ is invariant. If one considers all the solutions starting from the ball of radius $R$, we show that after a finite time $T$ all these solutions are in $K_\eta$. Finally, using Stroock-Varadhan's Support Theorem~\ref{thm:SV} we prove that for all $z$ in the ball $\BB(0,R)$, $c_0$ is in the support of $P_T(z,\cdot)$. Now, we give a formal proof of these facts.

\medskip

\begin{figure}[ht]
\begin{center}
\input{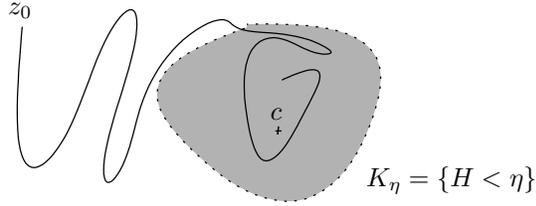}
\end{center}
\caption{Dynamic of the noise-free Hamiltonian system}
\label{fig:contraction}
\end{figure}

\medskip

For any $z\in\R^n$, let $T_z$ be the hitting time of $K_\eta$ starting from $z$, i.e.
$$T_z=\inf\left\{t\geq 0;\ z_t^z\in K_\eta\right\}.$$

We denote
$$T=\sup_{z\in \BB_R}T_z.$$
Since the hitting times of an open set are upper semicontinuous, we get
$$T<+\infty.$$

\medskip

\begin{proposition}\label{prop:prop}
With $T$ and $K_\eta$ defined as before, for all $z\in\BB(0,R)$,
$$P_T(z,K_\eta)>0.$$
\end{proposition}

\begin{proof}
Straightforward using the support Theorem~\ref{thm:SV} and the trivial control $\psi\equiv0$.
\end{proof}

\medskip

\begin{proof}[Proof of Theorem~\ref{thm:equilibre}]
First notice that we can use Lasalle's principle since
\begin{eqnarray*}
\partial_t H(z_t)&=&\sum_{i\in\VV}\partial_{q_i}H\dot{q}_i+\sum_{i\in\VV}\partial_{p_i}H\dot{p}_i\\
  &=&\sum_{i\in\VV}\partial_{q_i}Hp_i+\sum_{i\in\VV}p_i\left(-\partial_{q_i}H-\indicatrice{i\in\DD}p_i\right)\\
	&=&-\sum_{i\in\DD}p_i^2.
\end{eqnarray*}
Thus, the stability condition~\eqref{cond:stability} involves the convergence of the noise-free solutions to the equilibrium point. Let $\mu$ be an ergodic invariant measure, $\eg>0$. Since $\mu$ is non null, there exists a ball $\BB_R$ such that $\mu(\BB_R)>0$.
Since $H$ is continuous, there exists $\eta$ such that $K_\eta\subset \BB(c_0,\eg)$. Using the previous Proposition~\ref{prop:prop}, there exists $T$ such that for all $z\in \BB_R$,
$$P_T(z,K_\eta)>0.$$
Then, for any invariant measure $\mu$,
\begin{eqnarray*}
\mu\left(\BB(c_0,\eg)\right)&\geq&\mu\left(K_\eta\right)\\
	&=&P_T^\ast\mu\left(K_\eta\right)\\
	&=& \int_{\R^n}P_T\left(z,K_\eta\right)\ \mu(dz)\\
	&\geq&\int_{\BB_R}P_T\left(z,K_\eta\right)\ \mu(dz)\\
	&>&0.
\end{eqnarray*}
Finally, we have shown that $c_0\in \Supp\,\mu.$
\end{proof}

We can now prove uniqueness Theorem~\ref{thm:unique_lasalle}.
\begin{proof}[Proof of Theorem~\ref{thm:unique_lasalle}]
Let us recall (see e.g.~\cite{DP06}) that the set of invariant measures is the convex hull of the set of invariant ergodic measures. Moreover, if $(P_t)$ is asymptotically strong Feller at $c_0$, then $c_0$ belongs to at most one invariant ergodic measure~\cite[Theorem~3.16]{HM06}). But, if there are two invariant ergodic measures, the previous Theorem~\ref{thm:equilibre} ensures that $c_0$ is in both supports. That's impossible!\\
Finally, there is a unique ergodic invariant measure and, since the convex hull of a point is reduced to that point, diffusion~\eqref{eq:hcn} has at most one invariant measure.
\end{proof}

\subsubsection{The harmonic case}\label{section:harmonic_lasalle}
To improve our understanding of the stability condition defined below, we go back to the harmonic case where the pinning and the interaction potentials are quadratic.

\begin{theorem}
When the potentials are quadratic, the stability condition is equivalent to the asymmetry condition.
\end{theorem}

\begin{proof}
In this section, we want to prove that if $\dim\,\EE_{M,\DD}=n$, $z_t\equiv 0$ is the unique solution of the system
$$\left\{\begin{array}{rcl}
\dot{q}_i&=&p_i\\
\dot{p}_i&=&-q_i-\sum_{j\sim i}(q_i-q_j)\\
p_i\indicatrice{i\in\DD}&=&0.
         \end{array}\right.$$
Let us rewrite the equation associated to particle $i\in\DD$, $\langle z,e_{i+N}\rangle\equiv0$.
But, if we derive this equation with respect to the time parameter, as $p_i=0$ on $\DD$,
$$\sum_{j\sim i}q_j=constant.$$
Thus, we can write this equation (recall that $\Lg$ is symmetric) $\langle z,\Lg e_{i+N}\rangle\equiv0$.
By induction, we obtain for any $k\in\N$, $i\in\DD$, $t\in\R_+$,
$$\langle z,\Lg^k e_{i+N}\rangle=0.$$
Since $\dim\,\EE_{M,\DD}=n$, 
$$\mathrm{Span}\left\{\Lg^ke_{i+N},\ i\in\DD,\ k\in\N\right\}=\R^N,$$
thus,
$$\langle z,e_{i+N}\rangle\equiv 0,\ \forall i\in \VV.$$

Finally we obtain $q_i=constant,\ \forall i\in\{1,\ldots,N\}.$ Since $0$ is the unique solution of equation $H(z)=0$, we obtain
$$z\equiv 0.$$
\end{proof}

\subsection{Example}
The example presented in figure~\ref{fig:cex} is a heat conduction network that satisfies the stability condition~\eqref{cond:stability} but not Hörmander's condition. We suppose that the interaction potential $U$ satisfies $U(q)=q^4$ and that the pinning potential is quadratic, $V(q)=q^2$. Let us notice that all the particles are damped. Thus, stability condition is easily satisfied. We will see in Section~\ref{section:existence} that this system has an invariant measure. However, we are going to see below that Hörmander's condition is not satisfied.\\
First of all, $[X_0,\partial_{p_1}]=-\partial_{q_1}+\partial_{p_1}$, thus $\partial_{q_1}\in\Lie$.\\
We compute the Lie bracket between $X_0$ and $\partial_{q_1}$,
\begin{eqnarray*}
 [X_0,\partial_{q_1}]&=&\sum_{i\in\VV}\partial_{q_iq_1}^2H\partial_{p_i}\\
   &=&V''(q_1)\partial_{p_1}+U''(q_1-q_3)\partial_{p_1}+U''(q_1-q_4)\partial_{p_1}-\cdots\\
   &&\cdots-U''(q_3-q_1)\partial_{p_3}-U''(q_4-q_1)\partial_{p_4}.
       \end{eqnarray*}
So, $$U''(q_1-q_3)\partial_{p_3}+U''(q_1-q_4)\partial_{p_4}\in\Lie(z).$$
Computing the Lie bracket between the preceding quantity and $\partial_{q_1}$, we obtain
$$U'''(q_1-q_3)\partial_{p_3}+U'''(q_1-q_4)\partial_{p_4}\in\Lie(z).$$
With a similar computation we get, $U^{(4)}(q_1-q_3)\partial_{p_3}+U^{(4)}(q_1-q_4)\partial_{p_4}\in\Lie(z),$ thus
$$\partial_{p_3}+\partial_{p_4}\in\Lie(z).$$
We finally distinguish the following cases.
\begin{itemize}
\item If $q_1\neq q_3$ and $q_1\neq q_4$, Vandermonde's determinant does not vanish, hence
$$\dim\,\Lie(z)=12.$$
\item If $q_1=q_3$ and $q_1\neq q_4$ then $\partial_{p_4}\in\Lie(z)$, and the last Lie bracket gives $\partial_{p_3}\in\Lie(z)$, so
$$\dim\,\Lie(z)=12.$$
The same computations could be done if we consider the particles $2,\,5$ or $6$.

\item If $q_1=q_3=q_4=q_5=q_6=q$, we only get $\partial_{p_3}+\partial_{p_4}\in\Lie(z)$. Then, the Lie bracket with $X_0$ gives
$$\partial_{q_3}+\partial_{q_4}\in\Lie(z),$$
so
$$V''(q)\partial_{p_3}+V''(q)\partial_{p_4}\in\Lie.$$
Thus,
$$\dim\,\Lie(z)=10.$$
\end{itemize}
Finally, Hörmander's condition is not satisfied on the manifold defined by $\{z\in\R^n;\,q_1=\cdots=q_6\}$.

\section{Existence of the invariant measure}\label{section:existence}
Recall that $u$ (resp. $v$) denotes the degree of the polynomial $U$ (resp. $V$). Henceforth we assume that $u\geq v$, i.e. the interaction is stronger than the pinning. This section is devoted to the proof of the existence statement of Theorem~\ref{thm:existence}. The proof of this theorem is based on the arguments developed in~\cite{RBT02}. We give the outlines of the proof to emphasize the importance of the rigidity condition.

To prove the theorem we use the classical Krylov-Bogoliubov method.
\begin{lemma}\label{lem:kb}
Let us suppose that there exists a Lyapunov function $W$ such that
\begin{enumerate}
\item $\LL W(z)\leq CW(z)$ for a positive constant $C$,
\item there exists a time $t_0>0$ and a sequence $a_n\uparrow+\infty$ such that
$$\lim_{n\to\infty}\sup_{\{z;\,W(z)>a_n\}}\frac{P_{t_0}W(z)}{W(z)}=0.$$
\end{enumerate}
Then the semigroup $(P_t)$ has an invariant measure.
\end{lemma}

\begin{proof}
We divide the proof in two main steps.

First, let $K$ be the compact set $\{W(z)\leq a_n\}$, with $n$ large enough. Using the hypothesis, there exist $a\in(0,1),\,b$ such that
$$P_{t_0}W(z)\leq aW(z)+b\indicatrice{K}(z).$$
Thus, using the semigroup property, for any $n\in\N$,
\begin{eqnarray*}
P_{nt_0}W(z)&\leq& a^nW(z)+b+ab+\cdots+a^{n-2}b+a^{n-1}b\indicatrice{K}(z)\\
    &\leq&a^n W(z)+\frac{b}{1-a}.
\end{eqnarray*}
Since $\LL W(z)\leq CW(z)$, we get $P_tW(z)\leq e^{Ct}W(z)$ and choosing $n$ with $nt_0\leq t<(n+1)t_0$, we get
\begin{eqnarray*}
P_tW(z)&=&P_{nt_0}P_{t-nt_0}W(z)\\
  &\leq&e^{C(t-nt_0)}P_{nt_0}W(z)\\
  &\leq&e^{C(t-nt_0)}\left(a^nW(z)+\frac{b}{1-a}\right)\\
  &\leq&e^{Ct_0}\left(W(z)+\frac{b}{1-a}\right).
\end{eqnarray*}
Thus, $\sup_tP_tW(z)<+\infty.$

Second, let $z\in\R^n$. We consider the family of measures defined for any bounded continuous function $f$ by
$$\nu_t(f)=\tfrac{1}{t}\int_0^tP_sf(z)\,ds.$$
Using the first step, $\sup_t\nu_t(W)\leq\sup_tP_tW(z)<+\infty$. Thus, $(\nu_t)$ is tight and Krylov-Bogoliubov argument gives the existence of the invariant measure.
\end{proof}

To prove that the conditions of the previous lemma are satisfied in the heat conduction network setting, we use the Lyapunov function $W(z)=e^{\bg H(z)}$.
\begin{lemma}
Let $\bg\in\R$ such that $0<\bg<\max(T_i,\,i\in\partial\VV)^{-1}$. There exists a constant $C>0$ and $b>1$ such that
$$\frac{P_tW(z)}{W(z)}\leq e^{C\bg t\sum_iT_i}\E_z\left[e^{-C\int_0^t\sum_{i\in\DD}p_i^2\,ds}\right]^{1/b}.$$
\end{lemma}

We recall briefly the proof of this lemma.
\begin{proof}
Let $\bg\in\left(0,\max\left(T_i,\,i\in\partial\VV\right)^{-1}\right)$. Using Itô's formula,
\begin{eqnarray*}
H(Z_t^z)&=&H(z)+\sum_{i\in\partial\VV}T_it-\int_0^t\sum_{i\in\DD}p_i^2\,ds+\MM_t.
\end{eqnarray*}
But $(\MM_t)$ is a martingale with quadratic variation $\left[\MM,\MM\right]_t=2\int_0^t\sum_{i\in\partial\VV}T_ip_i^2\,ds.$ Thus,

\begin{eqnarray*}
\frac{P_tW(z)}{e^{\bg H(z)}}&=&\E_z\left[e^{\bg (H(Z_t)-H(z))}\right]\\
  &=&\E_z\left[e^{\bg \MM_t-a\frac{\bg^2}{2}[\MM,\MM]_t}e^{a\frac{\bg^2}{2}[\MM,\MM]_t+\bg\sum_{i\in\partial\VV}T_it-\bg\int_0^t\sum_{i\in\DD}p_i^2\,ds}\right]\\
\end{eqnarray*}
The inequality is obtained using Hölder inequality and exponential martingales.
\end{proof}

To show the semigroup property in Krylov-Bogoliubov Lemma~\ref{lem:kb}, we use a scaling argument to reduce our problem to a noise-free dynamics.

Let $(E_n)$ be an energy sequence growing to $+\infty$. We use the scaled Hamiltonian
$$H_n(q,p)=\sum_{i\in\VV}\frac{p_i^2}{2}+\tfrac{1}{E_n}V\left(E_n^{1/u}q_i\right)+\tfrac{1}{2E_n}\sum_{j\sim i}U\left(E_n^{1/u}(q_j-q_i)\right).$$
Let us consider the following stochastic differential equation
\begin{equation*}
\left\{\begin{array}{rcl}
       dq_i&=&\partial_{p_i}H_n\,dt\\
       dp_i&=&-\partial_{q_i}H_n\,dt-\indicatrice{i\in\DD}E_n^{1/u-1/2}\,p_i\,dt+\indicatrice{i\in\partial\VV}E_n^{1/2u-3/4}dB_i.
       \end{array}\right.
\tag{$\SS_n$}\label{eq:scaling}
\end{equation*}

The solutions of this system converge to the solutions of the Hamiltonian system described by the limit Hamiltonian (see~\cite[Section~5]{C07})
$$H_\infty(q,p)=\sum_{i\in\partial\VV}\frac{p_i^2}{2}+\indicatrice{u=v}|q_i|^v+\tfrac{1}{2}\sum_{j\sim i}(q_j-q_i)^u.$$

Thus, we consider the noise-free system, for any $i\in\VV$,
\begin{equation*}
\left\{\begin{array}{rcl}
       \dot{q}_i&=&\partial_{p_i}H_\infty\\
       \dot{p}_i&=&-\partial_{q_i}H_\infty-\indicatrice{i\in\DD}\,p_i.
       \end{array}\right.
\tag{$\SS_\infty$}\label{eq:sans_bruit}
\end{equation*}

To conclude, we need the following \emph{rigidity condition}. The differential system~\eqref{eq:hcn} is said to be rigid if
\begin{equation}\label{cond:rigidite}
\left\{\begin{array}{rcl}
       \dot{q}_i&=&\partial_{p_i}H_\infty\\
       \dot{p}_i&=&-\partial_{q_i}H_\infty\\
       p_i\indicatrice{i\in\DD}&=&0.
       \end{array}\right.
\end{equation}
has a unique solution given by $z_t\equiv c_0$.

\begin{lemma}
If the system satisfies the rigidity condition, then for any $z_0\in\R^n$ such that $H_\infty(z_0)=1$, for any solution of~\eqref{eq:sans_bruit} starting from $z_0$, for any $t>0$,
$$\int_0^t\sum_{i\in\DD}p_i^2\,ds>0.$$
\end{lemma}

\begin{proof}
Let us suppose that there exists $t>0$ such that
$$\int_0^t\sum_{i\in\DD}p_i^2\,ds=0.$$
Let us recall that we supposed that $H$ is strictly convex and reaches its minimum at $c_0\in\R^n$. Since the rigidity condition~\eqref{cond:rigidite} is satisfied, $z_t=c_0$ and we get $H_\infty(z_0)=H_\infty(c_0)=0$.
We have thus obtained a contradiction.
\end{proof}

The existence Theorem~\ref{thm:existence} is a consequence of the following lemma. The proof can be found in~\cite[Lemma~5.2]{C07}.
\begin{lemma}
Let $Z_n$ be a solution of the scaled system~\eqref{eq:scaling} starting from $z_n$ with $H(z_n)\to\infty$. Then, there exists a subsequence $(z_{n_k})_k$ such that for any $C>0,\,t>0$,
$$\lim_{k\to\infty}\E_{z_{n_k}}\left[e^{-C\int_0^t\sum_{i\in\DD}p_i^2\,ds}\right]=0.$$
\end{lemma}

\begin{remark} 
 Contrary to the results obtained for a chain, we cannot prove using this method that the convergence speed to the invariant measure is exponential. Indeed, this rate was obtained via compactness properties that we could not generalize. However, the results of~\cite{HM07} are still valid~:~when the interaction potential is quadratic and the pinning potential is at least of degree $4$, $0$ is in the essential spectrum of the extension of the generator $\LL$ to the space $L^2(e^{-\bg H}\,dp\,dq)$, with $\bg<\min(1/T_i,\, i\in\partial\VV)$.
\end{remark}

\appendix
\section{The Asymptotic Strong Feller property}\label{section:asf}
In this section we prove that the Asymptotic Strong Feller property is satisfied, if the potentials are harmonic, when the position of the damped particles is asymmetric, i.e. $\dim\,\EE_{M,\DD}=n$. When the potentials are convex polynomials, we expect that the diffusion~\eqref{eq:hcn} is Asymptotically Strong Feller at the equilibrium point $c_0$.

The Asymptotic Strong Feller property was introduced by M.~Hairer and J.~Mattingly in~\cite{HM06} to study the stochastic Navier-Stokes equation. We recall briefly the main definitions.
\begin{definition}[Totally separating system] An increasing sequence $(d_n)$ of pseudo-metrics is a \emph{totally separating system} if for all $x\neq y\in\R^n$,
$$d_n(x,y)\to 1.$$
\end{definition}

\begin{definition}[Asymptotic Strong Feller]
A semigroup $(P_t)$ is called \emph{asymptotically strong Feller} at $z\in\R^n$ if there exists a totally separating system of pseudo-metrics $(d_n)$ and a sequence $t_n>0$ such that
$$\lim_{\gg\to 0}\limsup_{n\to\infty}\sup_{y\in \BB(z,\gg)}\|P_{t_n}(z,\cdot)-P_{t_n}(y,\cdot)\|_{d_n}=0.$$
\end{definition}

\begin{proposition}[Gradient property, see~\cite{HM06}, Proposition~3.12] \label{prop:gradient}
Let $z\in\R^n$. Assume there exists a non-decreasing function $C:\R_+\to\R$ and two positive sequences $t_n\uparrow\infty,\ \dg_n\downarrow 0$ such that for all differentiable function $\phi$ with $\|\phi\|_\infty=\sup_z|\phi(z)|$ and $\|\nabla\phi(z)\|_\infty$ finite, for any $y$ in a neighbourhood of $z$,
$$|\nabla P_{t_n}\phi(y)|\leq C(\|y\|)\left(\|\phi\|_\infty+\dg_n\|\nabla\phi\|_\infty\right).$$
Then the semigroup is asymptotically strong Feller at $z$.

A corresponding totally separating system of pseudo-metrics is then
$$d_n(x,y)=1\wedge\frac{|x-y|}{\dg_n}.$$
\end{proposition}

\begin{proposition}
 When the potentials are harmonic and the graph is asymmetric, i.e. $\dim\,\EE_{M,\DD}=n$, the diffusion semigroup is asymptotically strong Feller at any point in $\R^n$.
\end{proposition}

\begin{proof}
We are going to use the stability of the matrix $M$ to prove that this property is true. First notice that for any unitary vector $\xi\in\R^n$,
\begin{eqnarray*}
D\left(P_t\phi\right)(z)\xi&=&\E\left[D\phi(Z_t^z)\xi\right]\\
	&=&\E\left[(D\phi)(Z_t^z)\circ (DZ_t^x)\xi\right]\\
	&=:&\E\left[(D\phi)(Z_t^z)\rg_t\right].\\
\end{eqnarray*}
If diffusion $(Z_t)$ satisfies the stochastic differential equation $dZ_t=f(Z_t)\, dt + \sg\, dB_t$, $\rg_t=D\phi(Z_t^z)\xi$ satisfies the differential equation (see~\cite[p.30]{B98}),
$$\left\{\begin{array}{ccl}
d\rg_t&=&\left(Df\right)(Z^x_t)\rg_t\ dt,\\
\rg_0&=&\xi.
         \end{array}\right.$$
Thus, for any function $\phi$ with bounded derivatives, we have the upper bound
$$\|\nabla P_t\phi(z)\|\leq \|\nabla\phi\|_\infty \E\left[\|\rg_t\|\right],$$
and it remains to show that $\E\left[\|\rg_t\|\right]\to 0$ when $t\to\infty$.

In the heat conduction network setting, we denote the Hessian of the potential energy
$$Hess(z)=\left(\partial_{q_iq_j}^2H(z)\right)_{1\leq i,j\leq N}.$$
Then, we have
$$\left\{\begin{array}{ccl}
\partial_t\rg_t&=&M(t,z) \rg_t, \\
\rg_0&=&\xi,
         \end{array}\right.$$
with $M(t,z)=\left(\begin{array}{cc}
                         0&I\\
                         -Hess(Z_t^z)&-I_\DD
                        \end{array}
\right)$.

When the potentials are quadratic, $M(t,z)=M$ is constant and since matrix $M$ is stable as soon as the network is asymmetric, we have $\rg_t\to 0$ for any $\xi$ if and only if $\dim\,\EE_{M,\DD}=n$. Thus, if the network is asymmetric, the diffusion is asymptotically strong Feller.
\end{proof}

When the potentials are no more harmonic, we have to study the behaviour of a non-autonomous linear differential equation.

\bibliography{hcn}

\begin{thebibliography}{EPRB99b}

\bibitem[Bas98]{B98}
Richard~F. Bass.
\newblock {\em Diffusions and elliptic operators}.
\newblock Probability and its Applications (New York). Springer-Verlag, New
  York, 1998.

\bibitem[BLLO08]{BLLO08}
Federico Bonetto, Joel~L. Lebowitz, Jani Lukkarinen, and Stefano Olla.
\newblock Heat conduction and entropy production in anharmonic crystals with
  self-consistent stochastic reservoirs.
\newblock \url{http://arxiv.org/abs/0809.0953}, 2008.

\bibitem[BO05]{BO05}
C{\'e}dric Bernardin and Stefano Olla.
\newblock Fourier's law for a microscopic model of heat conduction.
\newblock {\em J. Stat. Phys.}, 121(3-4):271--289, 2005.

\bibitem[Bol07]{B07}
François Bolley.
\newblock Separability and completeness for the wasserstein distance.
\newblock \url{http://www.ceremade.dauphine.fr/~bolley/wasserstein.pdf}, 2007.

\bibitem[Car07]{C07}
Philippe Carmona.
\newblock Existence and uniqueness of an invariant measure for a chain of
  oscillators in contact with two heat baths.
\newblock {\em Stochastic Process. Appl.}, 117(8):1076--1092, 2007.

\bibitem[DP06]{DP06}
Giuseppe Da~Prato.
\newblock {\em An introduction to infinite-dimensional analysis}.
\newblock Universitext. Springer-Verlag, Berlin, 2006.
\newblock Revised and extended from the 2001 original by Da Prato.

\bibitem[EH00]{EH00}
Jean-Pierre Eckmann and Martin Hairer.
\newblock Non-equilibrium statistical mechanics of strongly anharmonic chains
  of oscillators.
\newblock {\em Comm. Math. Phys.}, 212(1):105--164, 2000.

\bibitem[EPRB99a]{EPRB99b}
Jean-Pierre Eckmann, Claude-Alain Pillet, and Luc Rey-Bellet.
\newblock Entropy production in nonlinear, thermally driven {H}amiltonian
  systems.
\newblock {\em J. Statist. Phys.}, 95(1-2):305--331, 1999.

\bibitem[EPRB99b]{EPRB99a}
Jean-Pierre Eckmann, Claude-Alain Pillet, and Luc Rey-Bellet.
\newblock Non-equilibrium statistical mechanics of anharmonic chains coupled to
  two heat baths at different temperatures.
\newblock {\em Comm. Math. Phys.}, 201(3):657--697, 1999.

\bibitem[EZ04]{EZ04}
Jean-Pierre Eckmann and Emmanuel Zabey.
\newblock Strange heat flux in (an)harmonic networks.
\newblock {\em J. Statist. Phys.}, 114(1-2):515--523, 2004.

\bibitem[Hai05a]{H05}
Martin Hairer.
\newblock On the controllability of conservative systems.
\newblock \url{http://arxiv.org/abs/math-ph/0506064}, 2005.

\bibitem[Hai05b]{H08}
Martin Hairer.
\newblock A probabilistic argument for the controllability of conservative
  systems.
\newblock \url{http://arxiv.org/abs/math-ph/0506064}, 2005.

\bibitem[HM06a]{HM06}
Martin Hairer and Jonathan~C. Mattingly.
\newblock Ergodicity of the 2{D} {N}avier-{S}tokes equations with degenerate
  stochastic forcing.
\newblock {\em Ann. of Math. (2)}, 164(3):993--1032, 2006.

\bibitem[HM06b]{HM06sg}
Martin Hairer and Jonathan~C. Mattingly.
\newblock Spectral gaps in wasserstein distances and the 2d stochastic
  navier-stokes equations.
\newblock \url{http://www.citebase.org/abstract?id=oai:arXiv.org:math/0602479},
  2006.

\bibitem[HM07]{HM07}
Martin Hairer and Jonathan~C. Mattingly.
\newblock Slow energy dissipation in anharmonic oscillator chains.
\newblock \url{http://www.citebase.org/abstract?id=oai:arXiv.org:0712.3884},
  2007.

\bibitem[LS04]{LS04}
Raphaël Lefevere and Alain Schenkel.
\newblock Perturbative analysis of anharmonic chains of oscillators out of
  equilibrium.
\newblock {\em J. Statist. Phys.}, 115(5-6):1389--1421, 2004.

\bibitem[Mat07]{M07}
Jonathan~C. Mattingly.
\newblock Ergodicity of dissipative spdes.
\newblock Saint-Flour lecture notes, 2007.

\bibitem[MNV03]{MNV03}
Christian Maes, Karel Neto{\v{c}}n{\'y}, and Michel Verschuere.
\newblock Heat conduction networks.
\newblock {\em J. Statist. Phys.}, 111(5-6):1219--1244, 2003.

\bibitem[RB03]{RB03}
Luc Rey-Bellet.
\newblock Statistical mechanics of anharmonic lattices.
\newblock In {\em Advances in differential equations and mathematical physics
  (Birmingham, AL, 2002)}, volume 327 of {\em Contemp. Math.}, pages 283--298.
  Amer. Math. Soc., Providence, RI, 2003.

\bibitem[RB06]{RB06}
Luc Rey-Bellet.
\newblock Ergodic properties of {M}arkov processes.
\newblock In {\em Open quantum systems. II}, volume 1881 of {\em Lecture Notes
  in Math.}, pages 1--39. Springer, Berlin, 2006.

\bibitem[RBT02]{RBT02}
Luc Rey-Bellet and Lawrence~E. Thomas.
\newblock Exponential convergence to non-equilibrium stationary states in
  classical statistical mechanics.
\newblock {\em Comm. Math. Phys.}, 225(2):305--329, 2002.

\bibitem[Sas99]{S99}
Shankar Sastry.
\newblock {\em Nonlinear systems}, volume~10 of {\em Interdisciplinary Applied
  Mathematics}.
\newblock Springer-Verlag, New York, 1999.
\newblock Analysis, stability, and control.

\bibitem[SV72]{SV72}
Daniel~W. Stroock and Srinivasa R.~S. Varadhan.
\newblock On the support of diffusion processes with applications to the strong
  maximum principle.
\newblock In {\em Proceedings of the Sixth Berkeley Symposium on Mathematical
  Statistics and Probability (Univ. California, Berkeley, Calif., 1970/1971),
  Vol. III: Probability theory}, pages 333--359, Berkeley, Calif., 1972. Univ.
  California Press.

\bibitem[SZ70]{SZ70}
J.~Snyders and M.~Zakai.
\newblock On nonnegative solutions of the equation {$AD+DA=-C$}.
\newblock {\em SIAM Journal on Applied Mathematics}, 18(3):704--714, 1970.

\bibitem[Won79]{W79}
W.~Murray Wonham.
\newblock {\em Linear multivariable control: a geometric approach}, volume~10
  of {\em Applications of Mathematics}.
\newblock Springer-Verlag, New York, second edition, 1979.

\end{thebibliography}
\end{document}